\documentclass[12pt,reqno]{amsart}

\usepackage{amssymb,amsthm}
\usepackage{amsfonts,cmtiup,comment,stmaryrd}

\usepackage{mathrsfs,cmtiup}

\makeatletter
\def\blfootnote{\xdef\@thefnmark{}\@footnotetext}
\makeatother

\newtheorem{theorem}{Theorem}[section]
\newtheorem{lemma}[theorem]{Lemma}
\newtheorem{proposition}[theorem]{Proposition}
\newtheorem{corollary}[theorem]{Corollary}

\theoremstyle{definition}

\newtheorem{remark}[theorem]{Remark}
\newtheorem*{definition*}{Definition}

\let\leq=\leqslant
\let\geq=\geqslant
\numberwithin{equation}{section}

\begin{document}
\title{Compact groups with countable Engel sinks}

\author{E. I. Khukhro}
\address{Charlotte Scott Research Centre for Algebra, University of Lincoln, U.K., and \newline \indent  Sobolev Institute of Mathematics, Novosibirsk, 630090, Russia}
\email{khukhro@yahoo.co.uk}

\author{P. Shumyatsky}

\address{Department of Mathematics, University of Brasilia, DF~70910-900, Brazil}
\email{pavel@unb.br}

\keywords{Compact groups; profinite groups; pro-$p$ groups; finite groups; Lie ring method; Engel condition; locally nilpotent groups}
\subjclass[2010]{20E18, 20F19, 20F45, 22C05}

\begin{abstract}
An Engel sink of an element $g$ of a group $G$ is a set ${\mathscr E}(g)$ such that for every $x\in G$ all sufficiently long commutators $[...[[x,g],g],\dots ,g]$ belong to ${\mathscr E}(g)$.  (Thus, $g$ is an Engel element precisely when we can choose ${\mathscr E}(g)=\{ 1\}$.) It is proved that if every element of a compact (Hausdorff) group $G$ has a countable (or finite) Engel sink, then $G$ has a finite normal subgroup $N$ such that $G/N$ is locally nilpotent. This settles a question suggested by J.~S.~Wilson.
\end{abstract}
\maketitle

\section{Introduction}

A group $G$ is called an Engel group if for every $x,g\in G$ the equation $[x,\,{}_{n} g]=1$ holds for some $n=n(x,g)$ depending on $x$ and $g$.
Henceforth, we use the left-normed simple commutator notation
$[a_1,a_2,a_3,\dots ,a_r]:=[...[[a_1,a_2],a_3],\dots ,a_r]$ and the abbreviation $[a,\,{}_kb]:=[a,b,b,\dots, b]$ where $b$ is repeated $k$ times. A group is said to be locally nilpotent if every finite subset generates a nilpotent subgroup. Clearly, any locally nilpotent group is an Engel group. Wilson and Zelmanov \cite{wi-ze} proved the converse for profinite groups: any Engel profinite group is locally nilpotent. Later Medvedev \cite{med} extended this result to Engel compact  groups. (Henceforth by compact groups we mean compact Hausdorff groups.)

Generalizations of Engel groups can be defined in terms of Engel sinks.

\begin{definition*} \label{d}
 An \textit{Engel sink} of an element $g$ of a group $G$ is a set ${\mathscr E}(g)$ such that for every $x\in G$ all sufficiently long commutators $[x,g,g,\dots ,g]$ belong to ${\mathscr E}(g)$, that is, for every $x\in G$ there is a positive integer $n(x,g)$ such that
 $[x,\,{}_{n}g]\in {\mathscr E}(g)$ for all $n\geq n(x,g).
 $
 \end{definition*}
 \noindent (Thus, $g$ is an Engel element precisely when we can choose ${\mathscr E}(g)=\{ 1\}$, and $G$ is an Engel group when we can choose ${\mathscr E}(g)=\{ 1\}$ for all $g\in G$.)

Earlier we considered in \cite{khu-shu} compact groups $G$ in which  every element has a finite Engel sink and proved the following theorem.

\begin{theorem}[{\cite[Theorem~1.1]{khu-shu}}]\label{t4.1}
If every element of a compact group $G$ has a finite Engel sink, then $G$ has a finite normal subgroup $N$ such that $G/N$ is locally nilpotent.
\end{theorem}

In addition, a similar result of quantitative nature was also proved for finite groups.

In discussions with John Wilson a question was raised whether the condition on Engel sinks can be weakened to being countable. (By ``countable'' we mean ``finite or  denumerable''.) In this paper we answer this question in the affirmative.

\begin{theorem}\label{t1}
Suppose that $G$ is a compact group in which every element has a countable Engel sink. Then $G$ has a finite normal subgroup $N$ such that $G/N$ is locally nilpotent.
\end{theorem}

Thus, if all elements of a compact group have at most countable Engel sinks, then in fact all Engel sinks can be chosen to be  finite (and contained in the same finite normal subgroup). In Theorem~\ref{t1} it also follows that there is a locally nilpotent subgroup of finite index -- just consider $C_G( N)$.

The proof uses the aforementioned Wilson--Zelmanov theorem for profinite groups. First the case of pro-$p$ groups  is considered, where Lie ring methods are applied including Zelmanov's theorem on Lie algebras satisfying a polynomial identity and generated by elements all of whose products are ad-nilpotent
\cite{ze92,ze95,ze17}. As we noted in \cite{khu-shu}, it is easy to see that if every element  of a pro-$p$ group has a finite Engel sink,  then the group is locally
nilpotent. But in the present paper, with countable Engel  sinks, the case of pro-$p$ groups requires substantial efforts.
Then the case of prosoluble groups is settled by using properties of Engel sinks in coprime actions and a Hall--Higman--type theorem.
The general case of profinite groups is dealt with by bounding the nonsoluble length of the group, which enables induction on this length. (We introduced the nonsoluble length in \cite{khu-shu131}, although bounds for nonsoluble length had been implicitly used in various earlier papers, for example, in the celebrated Hall--Higman paper \cite{ha-hi},
or in Wilson's  paper \cite{wil83}; more recently, bounds for the nonsoluble length were used in the study of verbal subgroups in finite and profinite groups \cite{dms1, 68, austral, khu-shu132}.)
Finally, the result for compact groups is derived with the use of the structure theorems for compact groups.

\section{Preliminaries}
In this section we recall some  notation and terminology and establish some general properties of Engel sinks in compact and profinite groups.

Our notation and terminology for profinite and compact groups is standard; see, for example,  \cite{rib-zal},  \cite{wil}, and \cite{hof-mor}.  A subgroup (topologically) generated by a subset $S$ is denoted by $\langle S\rangle$. Recall that centralizers are closed subgroups, while commutator subgroups $[B,A]=\langle [b,a]\mid b\in B,\;a\in A\rangle$ are the closures of the corresponding abstract commutator subgroups.

For a group $A$ acting by continuous automorphisms on a group $B$ we use the usual notation for commutators $[b,a]=b^{-1}b^a$ and commutator subgroups $[B,A]=\langle [b,a]\mid b\in B,\;a\in A\rangle$, as well as for centralizers $C_B(A)=\{b\in B\mid b^a=b \text{ for all }a\in A\}$
and $C_A(B)=\{a\in A\mid b^a=b\text{ for all }b\in B\}$.

We record for convenience the following simple lemma.

\begin{lemma} \label{l-fng}
Suppose that $\varphi$ is a continuous automorphism of a compact group $G$ such that
$G=[G,\varphi ]$. If $N$ is a normal subgroup of $G$ contained in $C_G(\varphi )$, then $N\leq Z(G)$.
\end{lemma}

\begin{proof}
The centralizer $C_{G\langle \varphi\rangle}(N)$ of $N$ in the semidirect product $G\langle \varphi\rangle$ is a normal subgroup containing $\varphi$ by hypothesis. Hence,  $[g,\varphi]\in C_{G\langle \varphi\rangle}(N)$ for any $g\in G$. Since $C_{G\langle \varphi\rangle}(N)$ is a closed subgroup, the whole commutator subgroup $[G,\varphi ]$ is contained in $C_{G\langle \varphi\rangle}(N)$, and therefore, $G=[G,\varphi]\leq C_{G\langle \varphi\rangle}(N)$. This means that $N\leq Z(G)$.
\end{proof}

We denote by $\pi (k)$ the set of prime divisors of $k$, where $k$ may be a positive integer or a Steinitz number, and by $\pi (G)$ the set of prime divisors of the orders of elements of a (profinite) group $G$. Let $\sigma$ be a set of primes. An element $g$ of a group is  a $\sigma$-element if $\pi(|g|)\subseteq \sigma$, and a group $G$ is a $\sigma$-group if all of its elements are $\sigma$-elements. We denote by $\sigma'$ the complement of $\sigma$ in the set of all primes. When $\sigma=\{p\}$,  we write $p$-element, $p'$-element, etc.

Recall that a pro-$p$ group  is an inverse limit of finite $p$-groups, a pro-$\sigma $ group is an inverse limit of finite $\sigma$-groups, a pronilpotent group is an inverse limit of finite nilpotent groups, a prosoluble group is an inverse limit of finite soluble groups.

We denote by  $\gamma _{\infty}(G)=\bigcap _i\gamma _i(G)$ the intersection of the lower central series of a group $G$. A profinite group $G$ is pronilpotent if and only if $\gamma _{\infty}(G)=1$.

Profinite groups have Sylow $p$-subgroups and satisfy analogues of the Sylow theorems.  Prosoluble groups satisfy analogues of the theorems of Hall and Chunikhin on Hall $\pi$-subgroups and Sylow bases. We refer the reader to the corresponding chapters in \cite[Ch.~2]{rib-zal} and \cite[Ch.~2]{wil}. We add a simple folklore lemma.

\begin{lemma}\label{l-prosol-by-prosol}
A profinite group $G$ that is an extension of a prosoluble group $N$ by a prosoluble group $G/N$ is prosoluble.
\end{lemma}

\begin{proof}
The quotient $G/M$ of $G$ by an open normal subgroup $M$ is an extension of a finite soluble normal subgroup $NM/M$ by a finite soluble group $G/NM$, and therefore $G/M$ is soluble. Hence $G$ is prosoluble.
\end{proof}

We shall use several times the following well-known fact, which is straightforward from the Baire Category Theorem (see \cite[Theorem~34]{kel}).

\begin{theorem}\label{bct}
If a compact Hausdorff group is a countable union of closed subsets, then one of these subsets has non-empty interior.
\end{theorem}

Here is one fact following from this theorem.

\begin{lemma}\label{l-unci}
If $H$ is a closed subgroup of a compact group $G$, then the index of $H$ in $G$ is either finite or uncountable.
\end{lemma}

\begin{proof}
Suppose that the index of $H$ in $G$ is countable. Then $G$ is a countable union of closed cosets of $H$. By
Theorem~\ref{bct} one of these cosets has a non-empty interior, and therefore $H$ is an open subgroup. Since $G$ is compact, the index of $H$ must be finite.
\end{proof}

We now establish a few general properties of Engel sinks.
Clearly, the intersection of two Engel sinks of a given element $g$ of a group $G$ is again an Engel sink of $g$, with the corresponding function $n(x,g)$ being the maximum of the two functions. Therefore, if $g$ has a \textit{finite} Engel sink, then $g$ has a unique smallest Engel sink. If $\mathscr E(g)$ is a smallest Engel sink of $g$, then the restriction of the mapping $x\mapsto [x,g]$ to $\mathscr E(g)$ must be surjective, which gives the following characterization.

\begin{lemma}[{\cite[Lemma~2.1]{khu-shu}}]\label{l-min}
If an element $g$ of a group $G$ has a finite Engel sink, then $g$ has a smallest Engel sink $\mathscr E (g)$ and for every $s\in \mathscr E (g)$ there is $k\in {\mathbb N}$ such that  $s=[s,\,{}_kg]$.
\end{lemma}

We now consider countable Engel sinks in compact groups.

\begin{lemma}\label{l1}
Suppose that an element $g$ of a compact
group $G$ has a countable Engel sink $\{s_1,s_2,\dots\}$.  Then there are positive integers $i,j$ and an open set $U$  such that
$$
[u,\,{}_ig]=s_j \qquad \text{for all}\quad u\in U.
$$
If in addition $G$ is a profinite group, then there are  positive integers $i,j$ and a coset $Nb$ of an open normal subgroup $N$ such that
$$
[nb,\,{}_ig]=s_j \qquad \text{for all}\quad n\in N.
$$
\end{lemma}

\begin{proof}
We define the sets
$$
S_{kl}=\{x\in G\mid [x,\,{}_kg]=s_l\}.
$$
Note that each $S_{kl}$ is a closed subset of $G$. Then
$$
G=\bigcup _{k,l}S_{kl}
$$
by the definition of the Engel sink. By
Theorem~\ref{bct} one of these sets $S_{ij} $ contains an open subset $U$, as required. In the case of a profinite group this open subset contains a   coset $Nb$ of an open normal subgroup $N$.
\end{proof}

For profinite groups we can derive the following consequence of Lemma~\ref{l1}.

\begin{lemma}\label{l-engk} Suppose that an element $g$ of a profinite group $G$ has a countable Engel sink. Then there are positive integers $i,k$ and a coset $Nb$ of an open normal subgroup $N$  such that
$$
[[nb,\,{}_ia],a^k]=1 \qquad \text{for all}\quad n\in N.
$$
\end{lemma}

\begin{proof} Let
 $\{s_1,s_2,\dots\}$ be a countable Engel sink of $g$. By Lemma~\ref{l1}, there are  positive integers $i,j$ and a coset $Nb$ of an open normal subgroup $N$ such that
$$
[nb,\,{}_ia]=s_j\qquad \text{for all}\quad n\in N.
$$
Since $G/N$ is a finite group, the coset $Nb$ is invariant under conjugation by some power $a^k$. Then
$$
\begin{aligned}
s_j^{a^k}&=[b,\,{}_ia]^{a^{k}}=[b^{a^{k}},\,{}_ia]\\&=[nb,\,{}_ia]\quad \text{for some } n\in N\\
&=s_j.
\end{aligned}
$$
In other words, $a^{k}$ commutes with $s_j$, so that
\begin{equation*}
[[nb,\,{}_ia],a^{k}]=[s_j,a^k]=1\qquad \text{for all}\quad n\in N.\tag*{\qedhere}
\end{equation*}
\end{proof}

\begin{remark} If all Engel sinks in a group are at most countable, then this condition is inherited by every section of the group, and we shall use this property without special references. The same applies to a group in which all Engel sinks are finite.
\end{remark}

\section{Pronilpotent groups}

When $G$ is a pro-$p$ group, or more generally a pronilpotent group, the conclusion of the main Theorem~\ref{t1} is equivalent to $G$ being locally nilpotent, and this is what we prove in this section.

\begin{theorem}
\label{t2}
Suppose that $G$ is a pronilpotent group in which every element has a countable Engel sink. Then $G$ is locally nilpotent.
\end{theorem}

The bulk of the proof is about the case where $G$ is a pro-$p$ group. First we remind the reader of important Lie ring methods in the theory of pro-$p$ groups.

For a prime number $p $, the \textit{Zassenhaus $p $-filtration} of a group $G$ (also called the \textit{$p $-dimension series}) is defined by
$$
G_i=\langle g^{p ^k}\mid g\in \gamma _j(G),\;\, jp ^k\geqslant i\rangle \qquad\text{for}\quad i=1,2,\dots
$$
This is indeed a \textit{filtration} (or an \textit{$N$-series}, or a \textit{strongly central series}) in the sense that
\begin{equation}\label{e-fil}
[G_i,G_j] \leqslant G_{i+j}\qquad \text{for all}\quad i, j.
 \end{equation}
 Then the Lie ring $D_p (G)$ is defined with the additive group
$$
D_p(G)=\bigoplus _{i}G_i/G_{i+1},
$$
where the  factors $Q_i=G_i/G_{i+1}$ are additively written. The Lie product is defined on homogeneous elements $xG_{i+1}\in Q_i$, $yG_{j+1}\in Q_j$ via the group commutators by
$$
[xG_{i+1},\, yG_{j+1}] = [x, y]G_{i+j+1}\in Q_{i+j}
$$
and extended to arbitrary elements of $D_p(G)$ by linearity. Condition~\eqref{e-fil} ensures that this  product is well-defined, and group commutator identities imply that $D_p(G)$ with these operations is a Lie ring. Since all the factors $G_i/G_{i+1}$ have prime exponent~$p $, we can view $D_p(G)$ as a Lie algebra over the field of $p $ elements $\mathbb{F}_p $. We denote  by $L_p (G)$ the subalgebra generated by the first factor $G/G_2$. (Sometimes, the notation $L_p (G)$ is used for $D_p (G)$.)

A group $G$ is said to satisfy a \textit{coset identity} if there is a group word $w(x_1,\dots ,x_m)$ and cosets $a_1H,\dots ,a_mH$ of a subgroup $H\leqslant G$  such that $w(a_1h_1,\dots ,a_mh_m)=1$ for any $h_1,\dots ,h_m\in H$. We shall use the following  result of Wilson and Zelmanov \cite{wi-ze} about coset identities.

\begin{theorem}[{Wilson and Zelmanov \cite[Theorem~1]{wi-ze}}]\label{t-coset}
If a group $G$ satisfies a coset identity on cosets of a subgroup of finite index, then for every prime $p$ the Lie algebra $L_p (G)$ constructed with respect to the Zassenhaus $p $-filtration satisfies a polynomial identity.
\end{theorem}

 Theorem~\ref{t-coset} was used in the proof of the above-mentioned theorem on profinite Engel groups, which we state here for convenience.

\begin{theorem}[{Wilson and Zelmanov \cite[Theorem~5]{wi-ze}}]\label{t-wz}
Every profinite Engel group is locally nilpotent.
\end{theorem}

The proof of Theorem~\ref{t-wz} was based on the following deep result of Zelmanov \cite{ze92,ze95,ze17}, which is also used in our paper.

\begin{theorem}[{Zelmanov \cite{ze92,ze95,ze17}}]\label{tz}
 Let $L$ be a Lie algebra over a field and suppose that $L$ satisfies
a polynomial identity. If $L$ can be generated by a finite set $X$ such that every
commutator in elements of $X$ is ad-nilpotent, then $L$ is nilpotent.
\end{theorem}

We now consider pro-$p$ groups with countable Engel sinks.

\begin{proposition}
\label{pr-pro-p}
Suppose that $P$ is a finitely generated pro-$p$ group in which every element  has a countable Engel sink. Then $P$ is nilpotent.
\end{proposition}

\begin{proof}
We shall first prove that the Lie algebra $L_p(P)$ is nilpotent, using Theorem~\ref{tz}. The next two lemmas confirm that the hypotheses in Theorem~\ref{tz} are satisfied.

\begin{lemma}\label{l-ad}
The Lie algebra $L_p(P)$ is generated by finitely many elements all commutators in which are ad-nilpotent.
\end{lemma}

\begin{proof}
The image of the finite generating set  of $P$ in the first homogeneous component of the Lie algebra $L_p(P)$ is a finite set of generators of $L_p(P)$.  We claim that all commutators in these generators  are ad-nilpotent.
In fact, we prove that every homogeneous element $\bar a$ of $L_p(P)$
is ad-nilpotent. We may assume that $\bar a$ is the image of $a\in P$ in the corresponding factor of the Zassenhaus filtration $P_i/P_{i+1}$. We say that an element $g\in P_i\setminus P_{i+1}$  has \textit{degree}  $i$ with respect to this filtration.

By Lemma~\ref{l-engk} for any element $a\in P$ there are positive integers $i,s$ and  a coset $Nb$ of an open normal subgroup $N$ such that
$$
[[nb,\,{}_ia],a^{s}]=1\qquad \text{for all}\quad n\in N.
$$
Since $P$ is a pro-$p$ group, we can assume that $s$ is a power of $p$, so that
\begin{equation}\label{eq-ad}
[[nb,\,{}_ia],a^{p^k}]=1\qquad \text{for all}\quad n\in N.
\end{equation}

Slightly modifying the argument in \cite{wi-ze},  for generators $x,y,z,t$ of a free group we write
$$
[[xy,\,{}_i z],t]=[[x,\,{}_i z],t]\cdot [[y,\,{}_i z],t]\cdot v(x,y,z,t),
$$
where the word $v(x,y,z,t)$ is a product of commutators of weight at least $i + 3$, each of which involves $x$, $y$, $t$ and involves $z$ at least $i$ times. Substituting $x=n$, $y=b$, $z=a$, and $t=a^{p^k}$ and using \eqref{eq-ad} we obtain that
$$
[[n,\,{}_ia],a^{p^k}]=v(n,b,a,a^{p^k})^{-1}
\qquad \text{for all}\;\, n\in N.
$$
If $|P/N|=p^m$, then for any $g\in P$ we have $[g,\,{}_ma]\in N$, so that we also have
\begin{equation}\label{eq-ad3}
[[g,\,{}_{i+m}a],a^{p^k}]=v([g,\,{}_ma],b,a,a^{p^k})^{-1}.
\end{equation}
We claim that $\bar a$ is ad-nilpotent on $L_p(P)$ of index $i+m+p^k$.

Let $\delta  (u)$ denote the degree of an element $u \in P$ with respect to the Zassenhaus filtration.  It is known that
\begin{equation}\label{e-pd}
u^p\in P_{p\delta (u)}.
\end{equation}
  Furthermore, in $L_p(P)$ for the images of $u$ and $u^p$ in  $P_{\delta (u)}/P_{\delta (u)+1}$ and $P_{p\delta (u)}/P_{p\delta (u)+1}$, respectively, we have
\begin{equation}\label{e-pd2}
[x, \bar{u^p}]=[x,\,{}_p \bar{u}]
\end{equation}
(see, for example, \cite[Ch.~II, \S\,5, Exercise~10]{bou}).
By  \eqref{e-pd}  the degree of $v([g,\,{}_ma],b,a,a^{p^k})$ on the right of~\eqref{eq-ad3} is at least $\delta (b)+\delta (g)+(i+m+p^k)\delta (a)$,
   strictly greater than $d=\delta (g)+(i+m+p^k)\delta (a)$. This means that the image of the right-hand side of \eqref{eq-ad3}  in $P_d/P_{d+1}$ is trivial. At the same time, by \eqref{e-pd2} the image of the left-hand side of \eqref{eq-ad3}  in $P_d/P_{d+1}$ is equal to the image of $[g,\,{}_{i+m+p^k} a]$ in $P_d/P_{d+1}$, which is in turn equal to the  element    $[\bar g,\,{}_{i+m+p^k}\bar a]$ in $L_p(P)$.
 Thus, for the corresponding homogeneous elements of $L_p(P)$ we have
 $$
 [\bar g,\,{}_{i+m+p^k}\bar a]=0.
$$
 Since here $\bar g$ can be any homogeneous element, we obtain that $\bar a$ is ad-nilpotent of index $i+m+p^k$, as claimed.
\end{proof}

\begin{lemma}\label{l-PI}
The Lie algebra $L_p(P)$ satisfies a polynomial identity.
\end{lemma}

\begin{proof}
Let $\mathscr C$ be the family of all cosets $Nb$ where $N$ is a normal open subgroup of~$P$. Since $P$ is a finitely generated pro-$p$ group, the family  $\mathscr C$ is countable  \cite[Proposition~4.1.3]{wil}. For every $C\in \mathscr C$, let
$$
T_{C,i,k}=\{x\in P\mid [[g, \,{}_ix],x^{p^k}]=1 \text{ for all }g\in C\}.
$$
Note that the sets   $T_{C,i,k}$ are closed. By Lemma~\ref{l-engk} we have
$$
\bigcup _{c\in \mathscr C} T_{C,i,k}=P.
$$
Therefore by
Theorem~\ref{bct}  one of the $T_{C,i,k}$ contains an open subset. Thus, there are positive integers $i,k$ and cosets $Nb_1$, $Nb_2$   of an open normal subgroup $N$ such that
\begin{equation*}
[[y,\,{}_ix],x^{p^k}]=1 \qquad \text{for all}\;\,y\in Nb_1,\;\, x\in Nb_2.
\end{equation*}
Thus, $P$ satisfies  a coset identity and therefore the Lie algebra $L_p(P)$ satisfies a polynomial identity by Theorem~\ref{t-coset}.
\end{proof}

Now we can resume the proof of Proposition~\ref{pr-pro-p}. The two lemmas above together with Theorem~\ref{tz} show that  $L_p(P)$ is nilpotent.
The nilpotency of the Lie algebra $L_p (P)$  of the finitely generated pro-$p $ group $P$  implies that $P$ is a $p $-adic analytic group. This result goes back to Lazard~\cite{laz}; see also \cite[Corollary~D]{sha}.
Furthermore, by a theorem of Breuillard and Gelander \cite[Theorem~8.3]{br-ge}, a  $p $-adic analytic group satisfying a coset identity on cosets of a subgroup of finite index is soluble.

Thus, $P$ is soluble, and we prove that $P$ is nilpotent by induction on the derived length of $P$.  By induction hypothesis, $P$ has an abelian normal subgroup $U$ such that $P/U$ is  nilpotent. We aim to show that $P$ is an Engel group. Since $P/U$ is nilpotent, it is sufficient to show that every element $a\in P$ is an Engel element in the product $U\langle a\rangle$.

Applying Lemma~\ref{l1} to $U\langle a\rangle$ we obtain a coset $Nb$ of an open normal subgroup $N$ of $U\langle a\rangle$, a positive integer $i$, and some element $s$ of the Engel sink of $a$ such that
$$
[nb,\,{}_ia]=s\qquad \text{for all}\;\,n\in N.
$$
  Since $[a^iun,a]=[un,a]$ for any $u\in U$, $n\in N$, we can assume that $b\in U$. Then for any $m\in U\cap N$ we have
$$
s=[mb,\,{}_ia]=[m,\,{}_ia]\cdot [b,\,{}_ia]=[m,\,{}_ia]\cdot s,
$$
since $U$ is abelian. Hence, $[m,\,{}_ia]=1$ for any $m\in U\cap N$. Since $U\cap N$ has finite index in $U$ and $U\langle a\rangle$ is a pro-$p$ group, it follows that $a$ is an Engel element of  $U\langle a\rangle$.

Thus, $P$ is an Engel group and therefore, being a finitely generated pro-$p$ group,  $P$ is nilpotent by
Theorem~\ref{t-wz}.
\end{proof}

\begin{proof}[Proof of Theorem~\ref{t2}] By
Theorem~\ref{t-wz},  it is sufficient to prove that $G$ is an Engel group.
For each prime~$p$, let $G_p$ denote the Sylow $p$-subgroup of $G$, so that $G$ is a Cartesian product of the~$G_p$, since $G$ is pronilpotent. Given any two elements $a,g\in G$, we write $g=\prod _pg_p$ and $a=\prod _pa_p$,  where $a_p,g_p\in G_p$. Clearly, $[g_q,a_p]=1$ for $q\ne p$.

By Lemma~\ref{l-engk}, for the element $a\in G$ there are positive integers $i,k$ and a coset $Nb$ of an open normal subgroup $N$  such that
\begin{equation}\label{e-eng2}
[[nb,\,{}_ia],a^{k}]=1\qquad \text{for all}\quad n\in N.
\end{equation}
Let $l$ be the (finite) index of $N$ in $G$. Then $N$ contains all Sylow $q$-subgroups of $G$ for $q\not\in \pi (l)$. Hence we can choose $b$ to be a $\pi(l)$-element. Let $\pi=\pi (l)\cup \pi (k)$; note that $\pi$ is a finite set of primes.

We claim that
$$
[g_q,\,{}_{i+1}a_q]=1\qquad \text{for }q\not\in \pi.
$$
 Indeed, since $b$ commutes with elements of $G_q$ and $G_q\leq N$,  by \eqref{e-eng2} we have
\begin{equation}\label{eq-engq2}
\begin{aligned}
 1=[g_qb,\,{}_ia],a^{k}]
 & = [[g_q,\,{}_ia],a^{k}]\cdot  [[b,\,{}_ia],a^{k}]\\
 & =[[g_q,\,{}_ia],a^{k}]\\
 &= [[g_q,\,{}_ia_q],a_q^{k}].
 \end{aligned}
\end{equation}
 Thus, $a_q^{k}$ centralizes $[g_q,\,{}_ia_q]$. Since $k$ is coprime to $q$, we have $ \langle a_q^{k}\rangle =\langle a_q\rangle$. Therefore \eqref{eq-engq2} implies that $[[g_q,\,{}_ia_q],a_q]=1$, as claimed.

 For every prime $p$ the group $G_p$ is locally nilpotent by Proposition~\ref{pr-pro-p},  so there is $k_p$ such that $[g_p,\,{}_{k_p}a_p]=1$.  Now for $m=\max\{i+1, \max_{p\in \pi} \{k_p\}\}$ we have $[g_p,\,{}_{m}a_p]=1$ for all $p$, which means that  $[g,\,{}_{m}a]=1$.
 Thus, $G$ is an Engel group and therefore it is locally nilpotent by
 Theorem~\ref{t-wz}.
\end{proof}

\section{Coprime actions}

In this section,  first we list several profinite  analogues of the properties of coprime automorphisms of finite groups; we prove some of them in those cases where we could not find a convenient reference to the literature. Then we prove several lemmas on coprime automorphisms in relation to Engel sinks.

 If $\varphi$  is an automorphism of a finite group $H$ of coprime order, that is, such that $(|\varphi |,|H|)=1$, then  we say for brevity that $\varphi$ is a coprime automorphism of~$H$. This definition is extended to profinite groups as follows.
We say that $\varphi$ is a \textit{coprime automorphism}  of a profinite group $H$  meaning that a procyclic group $\langle\varphi\rangle$ faithfully acts on $H$ by continuous automorphisms   and $\pi (\langle \varphi\rangle)\cap \pi (H)=\varnothing$. Since the semidirect product $H\langle \varphi\rangle$ is also a profinite group, $\varphi$ is a coprime automorphism of $H$
 if and only if for every open normal $\varphi$-invariant subgroup $N$ of $H$ the automorphism (of finite order) induced by $\varphi$ on $H/N$ is a coprime automorphism.
The following folklore lemma follows from the Sylow theory for profinite groups and an analogue of the Schur--Zassenhaus theorem.

\begin{lemma}\label{l-inv}
If $\varphi$ is a coprime automorphism of a profinite group $G$, then for every prime $q\in \pi (G)$ there is a $\varphi$-invariant Sylow $q$-subgroup of $G$. If $G$ is in addition prosoluble, then for every subset $\sigma\subseteq \pi (G)$ there is a $\varphi$-invariant Hall $\sigma$-subgroup of~$G$.
\end{lemma}

\begin{proof}
Let $Q$ be a Sylow $q$-subgroup of $G$. By the analogue of Frattini argument \cite[Proposition 2.2.3(c)]{wil}, $G\langle \varphi\rangle=GN_{G\langle \varphi\rangle}(Q)$. By the analogue of the Schur--Zassenhaus theorem \cite[Proposition 2.3.3]{wil} applied to $N_{G\langle \varphi\rangle}(Q)$ and $ N_{G}(Q)$, there is a subgroup $K\leq N_{G\langle \varphi\rangle}(Q)$ such that $N_{G\langle \varphi\rangle}(Q)=KN_{G}(Q)$ and $K\cong \langle \varphi\rangle$; furthermore, $K$ and $\langle \varphi\rangle$ are conjugate in $G\langle \varphi\rangle=GK$. Thus, $\langle \varphi\rangle=K^x\leq N_{G\langle \varphi\rangle}(Q^x)$ for some $x\in G$, so that $\varphi$ normalizes the Sylow $q$-subgroup $Q^{x}$.

When $G$ is prosoluble, an analogue of the Frattini argument also holds for Hall subgroups because these are conjugate \cite[Corollary 2.3.7]{rib-zal}, and then the above proof  works in exactly the same manner.
\end{proof}

The following lemma is a special case of \cite[Proposition~2.3.16]{rib-zal}

\begin{lemma} \label{l-cover}
If $\varphi $ is a coprime automorphism of a profinite group $G$ and $N$ is a closed normal subgroup of $G$, then every fixed point of $\varphi$ in $G/N$ is an image of a fixed point of $\varphi$ in $G$, that is, $C_{G/N}(\varphi)=C(\varphi )N/N$.
\end{lemma}

As a consequence, we have the following.

\begin{lemma} \label{l-gff}
If $\varphi $ is a coprime automorphism of a profinite group $G$, then
$[[G,\varphi],\varphi]=[G,\varphi]$.
\end{lemma}

\begin{proof}
By Lemma~\ref{l-cover} we have $
G=C_G(\varphi)[G,\varphi]$, whence $
[G,\varphi]=[C_G(\varphi)[G,\varphi], \varphi]=[[G,\varphi], \varphi]$.
\end{proof}

\begin{lemma}\label{l-aac}
Let $\varphi$ be a coprime automorphism of finite order of a~profinite group~$G$. If $[g,\,{}_i\varphi]=1$ for $g\in G$ for some $i\in {\mathbb N}$, then $[g,\varphi]=1$.
\end{lemma}

\begin{proof}
Let $|\varphi|=n$. Due to the obvious induction it is sufficient to prove that if $[g,\varphi,\varphi]=1$, then $[g,\varphi]=1$. We have
$$
[g,\varphi]\cdot [g,\varphi]^{\varphi}\cdots [g,\varphi]^{\varphi^{n-1}}=g^{-1}g^{\varphi}\cdot (g^{-1})^{\varphi}g^{\varphi ^2} \cdot (g^{-1})^{\varphi ^2}g^{\varphi ^3}\cdots (g^{-1})^{\varphi^{n-1}}g^{\varphi ^n} =1
$$
and $[g,\varphi]\in C_G(\varphi )$, whence $[g,\varphi]^{n}=1$. Since $\pi (n)\cap \pi (G)=\varnothing$, it follows that $[g,\varphi]=1$, as required.
\end{proof}

The following useful lemma was proved for coprime automorphisms of finite nilpotent groups in \cite{rod-shu}.

\begin{lemma}[{\cite[Lemma~2.4]{rod-shu}}]\label{l-copr}
Let $\varphi$ be a coprime automorphism of a finite nilpotent
group $G$. Then any element $g\in G$ can be uniquely written in the form
$g=cu$, where $c\in C_G(\varphi)$ and $u=[v,\varphi ]$ for some $v\in G$.
 \end{lemma}

A similar result follows also for profinite groups, but we prefer to state more specialized lemmas following from Lemma~\ref{l-copr}.

 \begin{lemma}\label{l-copr1}
 Let $\varphi$ be a coprime automorphism of a pronilpotent
group~$G$. Then the restriction of the mapping
$$
\theta: x\mapsto [x,\varphi]
$$
to the set $K=\{ [g,\varphi]\mid g\in G\}$ is injective.
 \end{lemma}

\begin{proof}
 First consider the case where $G$ is finite (and nilpotent). Then $K=\theta (K)$. Indeed, every $g\in G$ is equal to $g=c[v,\varphi]$ for $c\in C_G(\varphi)$ and some $v\in G$ by Lemma~\ref{l-copr}. Hence,
$$
[g,\varphi]=[c[v,\varphi],\varphi]=[[v,\varphi],\varphi]\in \theta (K).
$$
 Thus, $\theta $ is surjective on the finite set $K$, and therefore also injective.

 The result for the general case of pronilpotent group $G$ follows: if we had $[u,\varphi,\varphi]=[v,\varphi,\varphi]$ for $[u,\varphi]\ne [v,\varphi]$, then the images of $[u,\varphi]$ and $[v,\varphi]$ would be different also in some finite quotient over an $\varphi$-invariant open normal subgroup, contrary to what was proved for finite groups.
\end{proof}

 \begin{lemma}\label{l-copr2}
 Let $\varphi$ be a coprime automorphism of a pronilpotent
group~$G$ with a countable Engel sink $\mathscr E(\varphi)$ in the semidirect product $G\langle \varphi\rangle$. Then the set $K=\{ [g,\varphi]\mid g\in G\}$ is a finite smallest Engel sink of $\varphi$ in the semidirect product $G\langle \varphi\rangle$.
 \end{lemma}

\begin{proof}
Since the mapping $\theta: x\mapsto [x,\varphi]$ is injective on the set $K$
by Lemma~\ref{l-copr1}, every mapping $\theta ^k: x\mapsto [x,\,{}_k\varphi]$ is also injective on $K$. Therefore for every $k\in {\mathbb N}$ and every element $s\in \mathscr E(\varphi )$ of the Engel sink $\mathscr E(\varphi )$ there is at most one element of the form $[g,\varphi]$ such that $[[g,\varphi], \,{}_k\varphi]=s$. Hence  for  every element $s\in \mathscr E(\varphi )$ there are at most countably many elements of the form $[g,\varphi]$ such that $[[g,\varphi], \,{}_j\varphi]=s$ for some $j$.

Since $ \mathscr E(\varphi )$ is countable and for every element $[g,\varphi]\in K$ there is $s\in  \mathscr E(\varphi )$ such that $[[g,\varphi], \,{}_j\varphi]=s$ for some $j$, the set $K$ is at most countable as a countable union of countable sets. It is well known that the set $K=\{ [g,\varphi]\mid g\in G\}$ is in a one-to-one correspondence with the set of (say, right) cosets of the centralizer $C_G(\varphi )$. But this set of cosets cannot be infinite countable by Lemma~\ref{l-unci}. Therefore it is finite, and so is the set $K$.

 The mapping $[g,\varphi ]\mapsto [g,\varphi,\varphi]$ is injective on $K$ by Lemma~\ref{l-copr1}, and therefore it is also surjective, since $K$ is finite. Hence this set is a smallest finite Engel sink of~$\varphi$.
 \end{proof}

 \begin{lemma}\label{l-copr3}
 Let $\varphi$ be a coprime automorphism of a pronilpotent
group~$G$. If  all elements of  the semidirect product $G\langle \varphi\rangle$ have countable Engel sinks, then $\gamma _{\infty} (G\langle \varphi\rangle)$ is finite  and  $\gamma _{\infty}(G\langle \varphi\rangle)= [G,\varphi]$.
 \end{lemma}

 \begin{proof}
The group $G$ is locally nilpotent by Theorem~\ref{t2}. By Lemma~\ref{l-copr2}, the set $K=\{ [g,\varphi]\mid g\in G\}$ is finite. Therefore the commutator subgroup $[G,\varphi ]=\langle K\rangle$  is nilpotent.
By Lemma~\ref{l-gff},
\begin{equation}\label{e-ff}
[[G,\varphi],\varphi]=[G,\varphi].
\end{equation}

Let $V$ be the quotient of  $[G,\varphi ]$ by its derived subgroup. For any $u,v\in V$ we have $[uv,\varphi ]=[u,\varphi][v,\varphi]$, since $V$ is abelian, and $[V,\varphi ]=V$ by \eqref{e-ff}. Hence $V$ consists of the images of elements of $K$, and therefore is finite. Then the nilpotent group $[G,\varphi]$ is also finite (see, for example, \cite[5.2.6]{rob}).

The quotient $G\langle \varphi\rangle/[G,\varphi]$ is obviously the direct product of the images of $G$ and $\langle \varphi\rangle$ and therefore is pronilpotent. Hence,  $\gamma _{\infty}(G\langle \varphi\rangle)\leq [G,\varphi]$, so $\gamma _{\infty}(G\langle \varphi\rangle)$ is finite. Since the set of commutators $\{ [g,\varphi]\mid g\in G\}$ is the smallest Engel sink of $\varphi$ by Lemma~\ref{l-copr2}, it follows that  $\gamma _{\infty}(G\langle \varphi\rangle)= [G,\varphi]$.
 \end{proof}

\section{Prosoluble groups}

In this section we prove Theorem~\ref{t1} for prosoluble groups.
First we consider the case of prosoluble groups of finite Fitting height. Recall that by Theorem~\ref{t2} any pronilpotent group with countable Engel sinks is locally nilpotent. Therefore, if $G$ is a profinite group with countable Engel sinks, then the largest pronilpotent normal subgroup $F(G)$  is also the largest locally nilpotent normal subgroup, and we call it the Fitting subgroup of $G$. Then further terms of the Fitting series are defined as usual by induction: $F_1(G)=F(G)$ and $F_{i+1}(G)$ is the inverse image of $F(G/F_i(G))$. A group has finite Fitting height if $F_k(G)=G$ for some $k\in {\mathbb N}$.

 \begin{proposition}\label{pr-height}
 Let $G$ be a prosoluble group of finite Fitting height. If  every element of  $G$ has a countable Engel sink, then $\gamma _{\infty} (G)$ is finite.
 \end{proposition}

\begin{proof}  It is sufficient to prove the result for the case of Fitting height 2. Then the general case will follow by induction on the Fitting height $k$ of $G$. Indeed, then $\gamma _{\infty}(G/ \gamma _{\infty}(F_{k-1}(G)))$ is finite, while $\gamma _{\infty}(F_{k-1}(G))$ is finite by the induction hypothesis, and as a result, $\gamma _{\infty}(G)$ is finite.

Thus, we assume that $G=F_2(G)$. By Theorem~\ref{t4.1}, it is sufficient to show that every element  $a\in G$ has a finite Engel sink.  Since $G/F(G)$ is locally nilpotent, an Engel sink of $a$ in $F(G)\langle a\rangle$ is also an Engel sink of $a$ in $G$.

For a prime $p$, let $P$ be a Sylow $p$-subgroup of $F(G)$, and write $a=a_pa_{p'}$, where $a_p$ is a $p$-element, $a_{p'}$ is a $p'$-element, and $[a_p,a_{p'}]=1$. Then $P\langle a_p\rangle$ is a normal Sylow $p$-subgroup of  $P\langle a\rangle$, on which $a_{p'}$ induces by conjugation a coprime automorphism. By Lemma~\ref{l-copr3} the subgroup $\gamma _{\infty}(P\langle a\rangle)=[P,a_{p'}]$ is finite. Since the pronilpotent group $P\langle a\rangle/\gamma _{\infty}(P\langle a\rangle)$ is locally nilpotent by Theorem~\ref{t2},  we can choose a finite smallest Engel sink $\mathscr E_p(a)\subseteq \gamma _{\infty}(P\langle a\rangle)$ of $a$ in $P\langle a\rangle$.

Note that
\begin{equation}\label{e-equiv}
   \text{if}\quad \mathscr E_p(a)=\{ 1\}, \quad\text{then}\quad \gamma _{\infty}(P\langle a\rangle)=1.
\end{equation}
Indeed, if  $\mathscr E_p(a)=\{ 1\}$, then, in particular, the image $\bar a$ of $a$ in $\langle a\rangle/C_{\langle a\rangle}([P,a_{p'}])$ is an Engel element of the finite group $[P,a_{p'}]\langle \bar a\rangle$ and therefore $\bar a$ is contained in its Fitting subgroup by  Baer's theorem \cite[Satz~III.6.15]{hup}. Then
$$
\gamma _{\infty}(P\langle a\rangle)=[P,a_{p'}]=[[P,a_{p'}],a_{p'}]=[[P,a_{p'}],\bar a_{p'}]=1.
$$

By Lemma~\ref{l-min}, for every $s\in \mathscr E_p(a)$ we have $s=[s,\,{}_ka]$ for some $k\in {\mathbb N}$, and then also
\begin{equation}\label{e-cycl}
s=[s,\,{}_{kl}a]\quad \text{for any}\;\, l\in {\mathbb N}.
\end{equation}

We claim that $\mathscr E_p(a)=\{1\}$ for all but finitely many primes $p$. Suppose the opposite, and $\mathscr E_{p_i}(a)\ne \{1\}$ for each prime $p_i$ in an infinite set of primes~$\pi$. Choose a nontrivial element $s_{p_i}\in \mathscr E_{p_i}(a)$ for every $p_i\in \pi$. For any subset $\sigma\subseteq \pi$, consider the (infinite) product
$$
s_{\sigma}=\prod _{p_j\in \sigma} s_{p_j}.
$$
Note that the elements $s_{p_j}$ commute with one another  belonging to different normal Sylow subgroups of $F(G)$.  If  $\mathscr E(a)$ is any Engel sink of $a$ in $G$, then for some $k\in {\mathbb N}$ the commutator $[s_{\sigma},\,{}_ka]$ belongs to $\mathscr E(a)$. Because of the properties \eqref{e-cycl}, all the components of $[s_{\sigma},\,{}_ka]$ in the Sylow $p_j$-subgroups of $F(G)$ for $p_j\in \sigma$ are non-trivial, while all the other components in Sylow $q$-subgroups for $q\not\in \sigma$ are trivial by construction. Therefore for different subsets $\sigma\subseteq \pi$ we thus obtain different elements of  $\mathscr E(a)$. The infinite set $\pi$ has continuum of different subsets, whence $\mathscr E(a)$ is uncountable, contrary to $a$ having a countable Engel sink by the hypothesis of the proposition.

Thus, for all but finitely many primes $p$ we have $\mathscr E_p(a)=\{1\}$, which is the same as $\gamma _{\infty}(P\langle a\rangle)=1$ by \eqref{e-equiv}. Therefore the subgroup
$$
\gamma _{\infty}(F(G)\langle a\rangle)=\prod_p\gamma _{\infty}(P\langle a\rangle)
$$
 is finite. The quotient $F(G)\langle a\rangle/\gamma _{\infty}(F(G)\langle a\rangle)$ is pronilpotent and therefore locally nilpotent by Theorem~\ref{t2}. Hence we can choose a finite Engel sink for $a$ in $G$ as a subset of $\gamma _{\infty}(F(G)\langle a\rangle)$.

Thus, every element of $G$ has a finite Engel sink,  and therefore $\gamma _{\infty}(G)$ is finite by Theorem~\ref{t4.1}.
\end{proof}

\begin{lemma}\label{l-finp}
Let $\varphi $ be a coprime automorphism
of a prosoluble group $G$ such that the set of primes $\pi (G)$ is finite. If every element of the semidirect product $G\langle \varphi\rangle$ has a countable Engel sink, then the subgroup $[G,\varphi ]$ is finite.
\end{lemma}

\begin{proof}
By Lemma~\ref{l-gff} we can assume that $G=[G,\varphi ]$. For every prime $q\in \pi (G)$ there is a $\varphi$-invariant Sylow $q$-subgroup $G_q$ of $G$ by Lemma~\ref{l-inv}. By Lemma~\ref{l-copr2} the set $\{[g,\varphi]\mid g\in G_q\}$ is finite. Since $\pi (G)$ is finite, there is an open normal subgroup $N$ of $G$ that intersects trivially with every set $\{[g,\varphi]\mid g\in G_q\}$, which implies that $\varphi$ centralizes every Sylow $q$-subgroup $N\cap G_q$ and therefore $[N,\varphi]=1$. Since $N$ is normal and $G=[G,\varphi ]$, we obtain $[N,G]=1$ by Lemma~\ref{l-fng}. Thus, $G/Z(G)$ is finite and, in particular, the Fitting height of $G$ is finite. Then $\gamma _{\infty} (G\langle \varphi\rangle)$ is finite by Proposition~\ref{pr-height}, and therefore $[G,\varphi]$ is also finite, since $[G,\varphi ]\leq  \gamma _{\infty} (G\langle \varphi\rangle)$ by Lemma~\ref{l-copr3}  applied to $G\langle \varphi\rangle/ \gamma _{\infty} (G\langle \varphi\rangle)$.
\end{proof}

\begin{lemma}\label{l-ha}
Let $a$ be a coprime automorphism of prime order $p$ of a prosoluble group $H$ such that $2\not\in\pi (H)$. If every element of the semidirect product $H\langle a\rangle$ has a countable Engel sink, then the subgroup $[H,a]$ is finite.
\end{lemma}

\begin{proof}
By Lemma~\ref{l-gff} we can assume that $H=[H,a]$. By Lemma~\ref{l1} there is a coset $bN$ of  an  open normal subgroup $N$ of $H$, a positive  integer $i$, and an element $s$ of an Engel sink of $a$ such that
\begin{equation}\label{e-ha}
[nb,\,{}_ia]=s \qquad \text{for all}\quad n\in N.
\end{equation}
Let $\pi=\pi(H/N)$. Then we can choose the element $b$ in \eqref{e-ha} in an $a$-invariant Hall $\pi$-subgroup $H_{\pi}$ of $H$ (which exists by Lemma~\ref{l-inv}). By Lemma~\ref{l-finp} the subgroup $[H_{\pi},a]$ is finite. Therefore $a$ has a finite smallest Engel sink $\mathscr E_{H_{\pi}}(a)$ in $H_{\pi}$. Replacing $i$ and $s$ in \eqref{e-ha} if necessary,  we can assume that $s\in \mathscr E_{H_{\pi}}(a)$ (for possibly a bigger integer $i$).

If $s=1$, then $[bN,a]=1$ by Lemma~\ref{l-aac}, and then $[N,a]=1$.  Since $N$ is normal and $H=[H,a ]$, we obtain $[N,H]=1$ by Lemma~\ref{l-fng}. Thus, $H/Z(H)$ is finite and, in particular, the Fitting height of $H$ is finite. Then $\gamma _{\infty} (H\langle a\rangle)$ is finite by Proposition~\ref{pr-height}, and $[H,a]\leq \gamma _{\infty} (H\langle a\rangle)$ by Lemma~\ref{l-copr3} applied to $H\langle a\rangle/ \gamma _{\infty} (H\langle a\rangle)$.

Therefore we can assume that $s$ is a nontrivial element of the finite smallest Engel sink $\mathscr E_{H_{\pi}}(a)$ contained in the finite group $[H_{\pi},a]$. By Lemma~\ref{l-min} we have $s=[s,\,{}_ka]$ for some $k\in {\mathbb N}$. Then the subgroup $S=\langle s^{\langle a\rangle}\rangle$ is finite of odd order coprime to $p=|a|$ and $S=[S,a]$.

If $x\in C_N(a)$, then taking the conjugates by $x$ of \eqref{e-ha} we obtain
$$
\begin{aligned}
s^x&=[b^x,\,{}_ia^x]\\&=[bn,\,{}_ia]\quad \text{for some }n\in N\\
&=s.
\end{aligned}
$$
Thus, $C_N(a)\leq C_H(s)$ and then also
\begin{equation}\label{e-cas}
C_N(a)\leq C_H(S).
\end{equation}

Recalling that $[H_{\pi},a]$ is finite, we choose an open normal subgroup $K$ of $H$ such that $K\cap [H_{\pi},a]=1$. Then $[K\cap H_{\pi},a]=1$.
We can of course choose $K\leq N$, so that we also have
 \begin{equation}\label{e-ha2}
[kb,\,{}_ia]=s \qquad \text{for all}\quad k\in K.
\end{equation}
 If $U$ is some $S\langle a\rangle $-invariant  $\pi$-section of $K$, then $a$ acts trivially on $U$ and therefore so does $S=[S,a]$ by Lemma~\ref{l-fng}.

We claim that $S$ also acts trivially on $S\langle a\rangle $-invariant  $\pi '$-sections of  $K$. If this is not the case, we can choose  an  $S\langle a\rangle $-invariant  elementary abelian $\pi '$-section $V$ of $K$ on which $S$ acts nontrivially and $S\langle a\rangle $ acts irreducibly. Let the bar denote the images of elements and subgroups of $S\langle a\rangle $ in the action on $V$. There is an $\bar a$-invariant Sylow $r$-subgroup $R$ of $\bar S$ such that $[R,\bar a]\ne 1$. Then $[R,\bar a]\langle \bar a\rangle$ is a $p$-soluble group without normal $p$-subgroups acting as a group of coprime automorphisms on $V$. Now the `non-modular' Hall--Higman--type results ensure that $C_V(a)\ne 1$. If $p\ne 2$, then $[R,\bar a]\langle \bar a\rangle$ is of odd order  and then $C_V(a)\ne 1$  by \cite[Theorem~IX.6.2]{hup-bla}.
If $p=2$, then one can choose an element in $R$ inverted by $a$ and then the situation is even simpler, for an action of a dihedral group of order $2r$. Since $a$ is a coprime automorphism of~$K$, every element of $ C_V(a)$ is an image of an element of $C_K(a)$ by Lemma~\ref{l-cover}.
By \eqref{e-cas} we obtain that $1\ne C_V(a)\leq  C_V(S\langle a\rangle)$. This contradicts the assumption that $S$ acts nontrivially and $S\langle a\rangle$ acts irreducibly on $V$.

As a result, we obtained that $S$ acts trivially on invariant $\pi$- and $\pi '$-sections of~$K$. Hence in any finite quotient of $K\langle s\rangle $ the image of $s$ is an Engel element  and therefore belongs to the Fitting subgroup by Baer's theorem \cite[Satz~III.6.15]{hup}. Then the subgroup  $[K,s]$ is pronilpotent and therefore $[K,s]\leq F(K)$.

We now consider the quotient $\bar H= H/F(K)$ denoting the images by bars.
Let $L= K\langle b, s, a\rangle$. Since $[\bar K,\bar s]=1$, while $b$ and $s$ belong to an $a$-invariant Hall $\pi$-subgroup $H_{\pi }$, the normal closure $\langle \bar s^{\bar L}\rangle$ is contained in $\bar H_{\pi }$.

Let tilde denote the images in the group $\tilde L=\bar L/\langle \bar s^{\bar L}\rangle$. By \eqref{e-ha2} we have
$$
[k\tilde b,\,{}_i\tilde a]=1 \qquad \text{for all}\quad k\in \tilde K.
$$
Since $\tilde a$ is a coprime automorphism, by Lemma~\ref{l-aac} we obtain $\tilde b\tilde K\subseteq C_{\tilde L}(\tilde a)$, and then also $\tilde K\leq C_{\tilde L}(\tilde a)$. In terms of the group $\bar H$ this means that $[\bar K, a]\leq \langle \bar s^{\bar L}\rangle \leq\bar H_{\pi }$. We also have $[\bar K, a]\leq \bar K$. Then
\begin{equation}\label{e-triv}
[\bar K, a,a]\leq [\langle \bar s^{\bar L}\rangle ,a]\cap \bar K \leq[\bar H_{\pi },a]\cap \bar K.
\end{equation}
 However, $[H_{\pi },a]\cap K=1$ by the choice of $K$, and $\bar H=H/F(K)$; therefore  the right-hand side of \eqref{e-triv} is also trivial. As a result we obtain
$[\bar K, a,a]=1$, whence   $[\bar K, a]=1$ by Lemma~\ref{l-aac}. Then also $[\bar K,\bar H]=1$ by Lemma~\ref{l-fng}, since $\bar H=[\bar H,a]$. Thus, $\bar H$ has a central subgroup of finite index, and therefore $\bar H$ has finite Fitting height. Since $\bar H=H/F(K)$, the whole group $H\langle a\rangle$ has finite Fitting height. By Proposition~\ref{pr-height} the group $\gamma _{\infty}(H\langle a\rangle)$ is finite. Since $a$ is a coprime automorphism of $H$, the subgroup $[H,a]$ is contained in $\gamma _{\infty}(H\langle a\rangle)$ by Lemma~\ref{l-copr3}, and therefore is also finite.
\end{proof}

Recall that any prosoluble group $G$ has a Sylow basis ---  a family of pairwise permutable Sylow $p_i$-subgroups $P_i$ of $G$, exactly one for each prime, and any two Sylow bases are conjugate (see  \cite[Proposition~2.3.9]{rib-zal}). The
basis normalizer (also known as the system normalizer) of such a Sylow basis in $G$ is $T=\bigcap _i N_G(P_i)$. If $G$ is a prosoluble group and $T$ is a basis normalizer in $G$, then $T$ is pronilpotent and $G =
\gamma _{\infty}(G)T$ (see \cite[Lemma~5.6]{rei}).

\begin{proposition}\label{pr-f}
If every element of a prosoluble group $G$ has a countable Engel sink, then $F(G)\ne 1$.
\end{proposition}

\begin{proof}

Let $\{P_i\mid i\in {\mathbb N}\}$ be a Sylow basis of $G$, where $P_i$ is a Sylow $p_i$-subgroup. Let
$$
T_1=\bigcap _iN_G(P_i)
$$
be a basis normalizer of $G$. Then  $T_1$ is pronilpotent and $G =
\gamma _{\infty}(G)T_1$. Note that $T_1\ne 1$, since $G$ is prosoluble.

The intersections $P_i\cap \gamma _{\infty}(G)$ are Sylow $p_i$-subgroups of $\gamma _{\infty}(G)$ forming a Sylow basis of $\gamma _{\infty}(G)$. Clearly, $T_1$ normalizes each of $P_i\cap \gamma _{\infty}(G)$.
Then
$$
T_2=\bigcap _iN_{\gamma _{\infty}(G)}(P_i\cap \gamma _{\infty}(G))
$$
is a basis normalizer of $\gamma _{\infty}(G)$, which is normalized by $T_1$. We know that $T_2$ is also pronilpotent, and $G=\gamma _{\infty}(G)T_1=\gamma _{\infty}(\gamma _{\infty}(G))T_2T_1$.

If $ \gamma _{\infty}(G)=1$, then there is nothing to prove. Otherwise, since $G$ is prosoluble, $\gamma _{\infty}(\gamma _{\infty}(G))\ne \gamma _{\infty}(G)$ and therefore the subgroup $T_2T_1$ is not pronilpotent, that is, $\gamma _{\infty}(T_2T_1)\ne 1$.
However,  $T_2T_1$ has Fitting height 2, and therefore $\gamma _{\infty}(T_2T_1)$ is a  finite group by Proposition~\ref{pr-height}. Therefore we can choose an element $a$ of prime order $p$ in $T_2$.

Since $T_2$ is a basis normalizer of $\gamma _{\infty}(G)$, for every prime $q$ there is an $a$-invariant Sylow $q$-subgroup $S_q$ of $\gamma _{\infty}(G)$.  Let $H$ be an $a$-invariant  Hall $\{p,2\}'$-subgroup of $\gamma _{\infty}(G)$ (which is simply a Hall $2'$-subgroup of $\gamma _{\infty}(G)$ in the case $p=2$). The following arguments include the cases where $H=1$ or/and $S_2=1$.

By Lemma~\ref{l-copr3} the subgroup $[S_2,a]$ is finite in the case $p\ne 2$. By Lemma~\ref{l-ha} the subgroup $[H,a]$ is finite. Therefore there is an open normal $a$-invariant subgroup $N$ of $\gamma _{\infty}(G)$ such that $N\cap [H,a]=1$, as well as $N\cap [S_2,a]=1$ in the case $p\ne 2$. Then $[N\cap H,a]=1$, as well as $[N\cap S_2,a]=1$  in the case $p\ne 2$.
Hence the subgroup $[N,a]$ is pronilpotent,  
and therefore,
$$
[N,a]\leq F(N)\leq F(\gamma _{\infty}(G))\leq F(G).
 $$
 Thus, the proposition is proved if $[N,a]\ne 1$.

If $[N,a]=1$, then also $[N,[\gamma _{\infty}(G),a]]=1$. Then $[\gamma _{\infty}(G),a]$ has a central subgroup of finite index and therefore has finite Fitting height. By Proposition~\ref{pr-height}, $\gamma _{\infty}([\gamma _{\infty}(G),a])$ is finite, and therefore $F([\gamma _{\infty}(G),a])\ne 1$ unless $[\gamma _{\infty}(G),a]=1$.
Since
$$
F([\gamma _{\infty}(G),a])\leq F(\gamma _{\infty}(G))\leq F(G),
$$
the proof is complete if  $[\gamma _{\infty}(G),a]\ne 1$.
 Finally, if $[\gamma _{\infty}(G),a]=1$, then $a$ is an Engel element since $G/\gamma _{\infty}(G)$ is locally nilpotent by Theorem~\ref{t2}. Then the normal subgroup  $[G,a]\langle a\rangle$ is pronilpotent by Baer's theorem \cite[Satz~III.6.15]{hup}, and $F(G)\ne 1$.
\end{proof}

We are now ready to prove the main result of this section.

\begin{theorem}\label{t3}
Suppose that $G$ is a prosoluble group in which every element has a countable Engel sink. Then $G$ has a finite normal subgroup $N$ such that $G/N$ is locally nilpotent.
\end{theorem}

\begin{proof}
By Theorem~\ref{t2} it is sufficient to prove that $\gamma _{\infty}(G)$ is finite.

By Proposition~\ref{pr-height} we obtain that $\gamma _{\infty}(F_2(G))$
is finite and the quotient $F_2(G)/\gamma _{\infty}(F_2(G))$ is locally nilpotent by Theorem~\ref{t2}. Then the subgroup $C=C_{F_2(G)}(\gamma _{\infty}(F_2(G)))$ has finite index in $F_2(G)$ and is locally nilpotent. Indeed, for any finite subset $S\subseteq C_{F_2(G)}(\gamma _{\infty}(F_2(G)))$ we have $\gamma _k(\langle S\rangle)\leq \gamma _{\infty}(F_2(G))$ for some $k$, and then
$$
\gamma _{k+1}(\langle S\rangle)=[\gamma _k(\langle S\rangle), \langle S\rangle]\leq [\gamma _{\infty}(F_2(G)), C]=1.
$$
As a normal locally nilpotent subgroup, $C$ is contained in $F(G)$. Hence, $F_2(G)/F(G)$ is finite.

We claim that the quotient $G/F(G)$ is finite. Let the bar denote the images in $\bar G=G/F(G)$. Then $F(\bar G)=\overline{F_2(G)}$ is finite by the above. There is an open normal subgroup $N$ of $\bar G$ such that $N\cap  F(\bar G)=1$. If $N\ne 1$, then $F(N)\ne 1$ by Proposition~\ref{pr-f}. But $F(N)\leq N\cap F(\bar G)=1$; hence we must have $N=1$, so $\bar G$ is finite.

Thus, $G/F(G)$ is finite, and therefore $G$ has finite Fitting height. By Proposition~\ref{pr-height} we obtain that $\gamma _{\infty}(G)$ is finite, as required.
\end{proof}

Here we also derive the following corollary for a virtually prosoluble group (that is, a group with a prosoluble open normal subgroup), which will be needed in the sequel.

\begin{corollary}\label{c-virt}
Suppose that $G$ is a virtually prosoluble group in which every element has a countable Engel sink. Then $G$ has a finite normal subgroup $N$ such that $G/N$ is locally nilpotent.
\end{corollary}

\begin{proof}
By Theorem~\ref{t2} it is sufficient to show that $\gamma _{\infty}(G)$ is finite.
By hypothesis, $G$ has an open normal prosoluble subgroup $H$.
 By Theorem~\ref{t3}, $\gamma _{\infty}(H)$ is finite. Therefore, passing to the quotient group,  we can assume that $\gamma _{\infty}(H)=1$ and the Fitting subgroup
$F(G)$ is open.

Since $G/F(G)$ is finite, we can use induction on $|G/F(G)|$.
       The basis of this induction includes the trivial case $G/F(G)=1$ when $\gamma _{\infty}(G)=1$. But the bulk of the proof deals with the case where $G/F(G)$ is a finite simple group.  If $G/F(G)$ is abelian, then $G$ has Fitting height 2 and  $\gamma _{\infty }( G)$ is finite by Proposition~\ref{pr-height} and the proof is complete.

Thus, suppose that $G/F(G)$ is a non-abelian finite simple group.
Let $p$ be a prime divisor of $|G/F(G)|$, and $g\in G\setminus F(G)$ an element of order $p^n$, where $n$ is either a positive integer or $\infty$ (so $p^n$ is a Steinitz number). Let $T$ be the Hall $p'$-subgroup of $F(G)$. By Lemma~\ref{l-copr3} the subgroup $[T , g]$ is finite.

Since $[T, g]$ is normal in $F(G)$, its normal closure $R=\langle [T, g]^G\rangle $ in $G$ is a product of finitely many conjugates and is therefore also finite.
Therefore it is sufficient to prove that $\gamma _{\infty }(G/R)$ is finite. Thus, we can assume that $R=1$. Note that then $[T, g^a]=1$ for any conjugate $g^a$ of $g$.

Choose a transversal $\{u_1,\dots, u_k\}$ of $G$ modulo $F(G)$.
 Let $G_1=\langle g^{u_1}, \dots ,g^{u_k}\rangle$. Clearly, $G_1F(G)/F(G)$ is generated by the conjugacy class of the image of $g$. Since $G/F(G)$ is simple, we have $G_1F(G)=G$. By our assumption, the Hall $p'$-subgroup $T$ of $F(G)$
 is centralized by all elements $g^{u_i}$. Hence,  $[G_1, T]=1$. Let $P$ be the Sylow $p$-subgroup of $F(G)$ (possibly, trivial). Then also $[PG_1, T]=1$, and therefore
 $$\gamma _{\infty }(G)=\gamma _{\infty }(G_1F(G))= \gamma _{\infty }(PG_1).$$

Let the bar denote images in $\bar G=G/P$. Note that $\gamma _{\infty}(\bar G)=\gamma _{\infty}(\bar G_1)$, while $F(\bar G)=\bar T$ and $\bar G/\bar T=\bar G_1\bar T/\bar T\cong F/F(G)$ is a non-abelian finite simple group. Hence, $\bar G=\gamma _{\infty}(\bar G_1)\bar T$. Therefore, since $[\gamma _{\infty}(\bar G_1), \bar T]=1$,
$$
\gamma _{\infty}(\bar G_1)=[\gamma _{\infty}(\bar G_1), \bar G_1]=[\gamma _{\infty}(\bar G_1),\gamma _{\infty}(\bar G_1)\bar T]=[\gamma _{\infty}(\bar G_1),\gamma _{\infty}(\bar G_1)].
$$
As a result, $\gamma _{\infty}(\bar G_1)\cap \bar{T}$ is contained both in the centre and the derived subgroup of $\gamma _{\infty}(\bar G_1)$, and therefore is isomorphic to a subgroup of the Schur multiplier of the finite group $\gamma _{\infty}(\bar G_1)/ (\gamma _{\infty}(\bar G_1)\cap \bar{T})\cong G/F(G)$. Since the Schur multiplier of a finite group is finite \cite[Hauptsatz~V.23.5]{hup}, we obtain that $\gamma _{\infty}(\bar G_1)\cap \bar{T}$ is finite. Since $\bar T$ is canonically isomorphic to $T$, it follows that
$$
\gamma _{\infty }(G)\cap T\cong\gamma _{\infty }(\bar G)\cap \bar T=\gamma _{\infty }(\bar G_1)\cap \bar T
$$
is also finite. Therefore  we can assume that $T=1$, in other words, that $F(G)$ is a $p$-group.

Since $G/F(G)$ is a non-abelian simple group, we can choose another prime $r\ne p$ dividing $|G/F(G)|$ and repeat the same arguments as above with $r$ in place of $p$. As a result, we reduce the proof to the case $F(G)=1$, where the result is obvious.

We now finish the proof of Corollary~\ref{c-virt} by induction on  $|G/F(G)|$. The basis of this induction where $G/F(G)$ is a simple group was proved above. Now suppose that $G/F(G)$ has  a nontrivial proper normal subgroup with full inverse image $N$, so that $F(G)<N\lhd G$. Since $F(N)=F(G)$, by induction applied to $N$ the group $\gamma _{\infty }(N)$ is finite. Since $N/\gamma _{\infty }(N)\leq F( G/\gamma _{\infty }(N))$,  by induction applied to $G/\gamma _{\infty }(N)$ the group $ \gamma _{\infty }(G/\gamma _{\infty }(N) )$ is also finite. As a result, $\gamma _{\infty }(G) $ is finite, as required.
\end{proof}

\section{Bounding the nonprosoluble length}

In  this section we approach
the case of profinite groups by obtaining bounds for the so-called nonprosoluble length. These bounds follow from the bounds for nonsoluble length of the corresponding finite quotients. We begin with the relevant definitions.

 The  \textit{nonsoluble length} $\lambda (H)$  of a finite group $H$ is defined as the minimum number of nonsoluble factors in a normal series in which every  factor  either is soluble or is a direct product of non-abelian simple groups. (In particular, the group is soluble if and only if its nonsoluble length is $0$.) Clearly, every finite group has a normal series with these properties, and therefore its nonsoluble length is well defined.  It is easy to see that the nonsoluble length $\lambda (H)$ is equal to the least positive integer $l$ such that there is a series of characteristic subgroups
\begin{equation*}
1=L_0\leqslant R_0 <  L_1\leqslant R_1<  \dots \leqslant R_{l}=H
\end{equation*}
in which each quotient $L_i/R_{i-1}$ is a (nontrivial) direct product of non-abelian simple groups, and each quotient $R_i/L_{i}$ is soluble (possibly trivial).

We shall use the following result of Wilson \cite{wil83}, which we state in the special case of $p=2$ using the terminology of nonsoluble length.

\begin{theorem}[{see \cite[Theorem~2*]{wil83}}]\label{t-wil83}
Let $K$ be a normal subgroup of a finite group~$G$. If a Sylow $2$-subgroup $Q$ of $K$ has a coset $tQ$ of exponent dividing $2^k$, then the nonsoluble length of $K$ is at most $k$.
\end{theorem}

We now turn to profinite groups. It is natural to say that a profinite group $G$ has finite \textit{nonprosoluble length} at most $l$ if $G$ has a normal series \begin{equation*}
1=L_0\leqslant R_0 <  L_1\leqslant R_1<  \dots \leqslant R_{l}=G
\end{equation*}
in which each quotient $L_i/R_{i-1}$ is a (nontrivial) Cartesian product of non-abelian finite simple groups, and each quotient $R_i/L_{i}$ is prosoluble (possibly trivial).
 As a special case of a general result in Wilson's paper \cite{wil83} we have the following.

\begin{lemma}[{see \cite[Lemma~2]{wil83}}]\label{l-nsl}
If, for some positive integer $m$, all continuous finite quotients of a profinite group~$G$ have nonsoluble length at most $m$, then $G$ has finite nonprosoluble length at most~$m$.
\end{lemma}

We are now ready to prove the main result of this section.

\begin{proposition}\label{pr-fnl}
Suppose that $G$ is a profinite group
in which every element has a countable Engel sink. Then $G$ has finite nonprosoluble length.
\end{proposition}

\begin{proof}
Let $ H=\bigcap G^{(i)} $
be the intersection of the derived series of $G$.
Then $H=[H,H]$. Indeed,
if $H\ne [H,H]$, then the quotient $G/[H,H]$ is a prosoluble group by Lemma~\ref{l-prosol-by-prosol}, whence $\bigcap G^{(i)}=H\leq [H,H]$, a contradiction.
Since the quotient $G/H$ is prosoluble, it is sufficient to prove the proposition for $H$. Thus, we can assume from the outset that $G=[G,G]$.

Let $T$ be a Sylow $2$-subgroup of $G$. By Theorem~\ref{t2} the group $T$ is locally nilpotent. Consider the subsets of the direct product $T\times T$
$$
S_{i}=\{(x,y)\in T\times T\mid \text{the subgroup }\langle x,y\rangle\text{ is nilpotent of class at most }i\}.
$$
Note that each subset $S_{i}$ is closed in the product topology of $T\times T$, because the condition defining $S_i$ means that all commutators of weight $i+1$ in $x,y$ are trivial. Since every $2$-generator subgroup of $T$ is nilpotent, we have
$$
\bigcup _iS_{i}=T\times T.
$$
By
Theorem~\ref{bct} one of the sets $S_i$ contains an open subset of $T\times T$. This means that there are cosets $aN$ and $bN$ of an open normal subgroup $N$ of $T$ and a positive integer $c$  such that
\begin{equation}\label{e-2nilp}
\langle x,y\rangle\text{ is nilpotent of class }c\text{ for any }x\in aN,\; y\in bN.
\end{equation}

Let $K$ be an open normal subgroup of $G$ such that $K\cap T\leq N$. If we replace $N$ by $K\cap T$, then \eqref{e-2nilp} still holds with the same $a,b$. Hence we can assume that $N$ is a Sylow $2$-subgroup of $K$.

We now apply the following general fact (which, for example, immediately follows from \cite[Lemma 2.8.15]{rib-zal}).

\begin{lemma}\label{l-dop} Let $G$ be a profinite group and $K$ a normal open subgroup of $G$. There exists a subgroup $H$ of $G$ such that
$G=KH$ and $K\cap H$ is pronilpotent.
\end{lemma}

Let $H$ be the subgroup given by this lemma for our group $G$ and subgroup $K$. Since $H$ is virtually pronilpotent and every element has a countable Engel sink, by Corollary~\ref{c-virt} the subgroup $\gamma _\infty(H)$ is finite. Recalling our assumption that $G=[G,G]$, we obtain
$$
G=[G,G]=\gamma _{\infty}(G)\leq \gamma _{\infty}(HK)\leq \gamma _{\infty}(H)K.
$$
Thus, $G=\gamma _{\infty}(H)K$, where $\gamma _{\infty}(H)$ is a finite subgroup.

Hence we can choose the coset representative $a$ satisfying \eqref{e-2nilp}
in a conjugate of a Sylow $2$-subgroup of $\gamma _{\infty}(H)$, and therefore having finite order, say, $|a|=2^n$.

For any $y\in bN$ the $2$-subgroup $\langle a,y\rangle$ is nilpotent of class at most $c$, while  $a^{2^n}=1$. Then
\begin{equation}\label{e-ay2nc}
[a,y^{2^{n(c-1)}}]=1.
 \end{equation}
 This follows from well-known commutator formulae (and for any $p$-group); see, for example, \cite[Lemma~4.1]{shu00}.

In particular,  for any $z\in N$ by using \eqref{e-ay2nc} we obtain
\begin{equation}\label{e-2eng2}
[z,\,{}_c y^{2^{n(c-1)}}]=[az,\,{}_c y^{2^{n(c-1)}}]=1,
\end{equation}
since $\langle az,y^{2^{n(c-1)}}\rangle$ is  a subgroup of  $\langle az,y\rangle$, which is nilpotent of class $c$  by \eqref{e-2nilp}.

Our aim is to show that there is a uniform bound, in terms of $|G:K|$, $c$, and $n$,  for the nonsoluble length of all finite quotients of $G$ by open normal subgroups.
Let $M$ be an open normal subgroup of $G$ and let the bar denote the images in $\bar G=G/M$. It is clearly sufficient to obtain a required bound for the nonsoluble length of $\bar K$.

Let $R_0$ be the soluble radical of $\bar K$, and $L_1$ the inverse image of the generalized Fitting subgroup of $\bar K/R_0$, so that
\begin{equation}\label{e-soc}
L_1/R_0=S_1\times S_2\times \dots\times S_k
\end{equation}
is a direct product of non-abelian finite simple groups. Note that $R_0$ and $L_1$ are normal subgroups of $\bar G$. The group $\bar G$ acting by conjugation induces a permutational action on the set $\{S_1,S_2,\dots ,S_k\}$. The kernel of the restriction of this permutational action to $\bar K$ is contained in the inverse image $R_1$ of the soluble radical of $\bar K/L_1$:
\begin{equation}\label{e-soc2}
\bigcap _iN_{\bar K}(S_i)\leq R_1.
\end{equation}
This follows from the validity of Schreier's conjecture on the solubility of the outer automorphism groups of non-abelian finite simple groups, confirmed by the classification of the latter, because $L_1/R_0$ contains its centralizer in $\bar K/R_0 $.

Let $e$ be the least positive integer such that $2^{e}\geq c$, and let $t= 2^{n(c-1)+e}$. We claim that for any   $y\in \bar b\bar N$ the element  $y^{2^t}$ normalizes each factor $S_i$ in \eqref{e-soc}.
Arguing by contradiction, suppose that the element $y^{2^t}$ has a nontrivial orbit on the set of the  $S_i$. Then the element $y^{2^{n(c-1)}}$ has an orbit of length $2^s\geq  2^{e+1}$ on this set; let $\{T_1,T_2,\dots , T_{2^s}\}$ be such an orbit cyclically permuted by $y^{2^{n(c-1)}}$.
Since non-abelian finite simple groups have even order (by the Feit--Thompson theorem \cite{fei-tho}) and the subgroups $S_i$ are subnormal in $\bar K/R_0$, each subgroup $S_i$ contains a nontrivial element of $\bar NR_0/R_0$. If $x$ is a nontrivial element of $T_1\cap  \bar NR_0/R_0$, then the commutator
$$
[x,\,{}_c \bar y^{2^{n(c-1)}}],
$$
written as an element of $T_1\times T_2\times \dots\times  T_{2^s}$, has a nontrivial component in $T_{c+1}$ since $2^s\geq  2^{e+1}>c$. This, however, contradicts~\eqref{e-2eng2}.

Thus, for any element  $y\in \bar b\bar N$ the power $y^{2^t}$ normalizes each factor $S_i$ in \eqref{e-soc}. Let $2^{d}$ be the highest power of $2$ dividing $|G:K|$, and let $u=\max\{t,d\}$. Then $y^{2^u}\in R_1$ by \eqref{e-soc2}, since $y^{2^u}\in \bar K$ and $y^{2^u}$ normalizes each $S_i$ in \eqref{e-soc} by the choice of $u$.

As a result, in the quotient $\bar G/R_1$ all elements of the coset $\bar b\bar NR_1/R_1$ of the Sylow $2$-subgroup $\bar NR_1/R_1$ of $\bar K/R_1$ have exponent dividing $2^u$. We can now apply  Theorem~\ref{t-wil83}, by which the nonsoluble length of $\bar K/R_1$ is at most $u$. Then the  nonsoluble length of $\bar K$ is at most $u+1$. Clearly, the nonsoluble length of $\bar G/\bar K$ is bounded in terms of  $|G:K|$.
As a result, since the number $u$ depends only on $|G:K|$, $n$, and $c$, the nonsoluble length of $\bar G$ is bounded in terms of these parameters only. Since this  holds for any quotient of the profinite group $G$ by a normal open subgroup,  the group $G$ has finite nonprosoluble length by Lemma~\ref{l-nsl}. This completes the proof of Proposition~\ref{pr-fnl}.
\end{proof}

\section{Profinite groups}

We are now ready to handle the general case of profinite groups using Corollary~\ref{c-virt} on virtually prosoluble groups and induction on the nonprosoluble length.
First we eliminate infinite Cartesian products of non-abelian finite simple groups.

\begin{lemma}\label{l-cart}
Suppose that $G$ is a profinite group that is a Cartesian product of non-abelian finite simple groups. If every element of $G$ has a countable Engel sink, then $G$ is finite.
\end{lemma}

\begin{proof}
Suppose the opposite: then $G$ is a Cartesian product of infinitely many non-abelian finite simple groups $G_i$ over an infinite set of indices $i\in I$. Every subgroup $G_i$ contains an element $g_i\in G_i$ with a nontrivial smallest Engel sink $\mathscr E(g_i)\ne \{1\}$. (Actually, any nontrivial element of $G_i$ has a nontrivial Engel sink, since an Engel element of a finite group belongs to its Fitting subgroup by Baer's theorem \cite[Satz~III.6.15]{hup}.) By Lemma~\ref{l-min}, for any $s\in \mathscr E(g_i)$ we have $s=[s,\,{}_{k}g_i]$ for some $k\in {\mathbb N}$, and then also
\begin{equation}\label{e-cycl2}
s=[s,\,{}_{kl}g_i]\quad \text{for any}\;\, l\in {\mathbb N}.
\end{equation}

For every $i$, choose a nontrivial element $s_{i}\in \mathscr E(g_i)\subseteq G_i$.
For any subset $J\subseteq I$, consider the (infinite) product
$$
s_{J}=\prod _{j\in J} s_{j}.
$$
Let $$
g=\prod _{i\in I} g_{i}.
$$
If  $\mathscr E(g)$ is any Engel sink of $g$ in $G$, then for some $k\in {\mathbb N}$ the commutator $[s_{J},\,{}_kg]$ belongs to $\mathscr E(g)$. Because of the properties \eqref{e-cycl2}, all the components of $[s_{J},\,{}_kg]$ in the factors $G_j$ for $j\in J$ are nontrivial, while all the other components in $G_i$ for $i\not\in J$ are trivial by construction. Therefore for different subsets $J\subseteq I$ we thus obtain different elements of  $\mathscr E(g)$. The infinite set $I$ has at least continuum of different subsets, whence $\mathscr E(g)$ is uncountable, contrary to $g$ having a countable Engel sink by the hypothesis.
\end{proof}

\begin{theorem}\label{t4}
Suppose that $G$ is a profinite group in which every element has a countable Engel sink. Then $G$ has a finite normal subgroup $N$ such that $G/N$ is locally nilpotent.
\end{theorem}

\begin{proof}
By Proposition~\ref{pr-fnl} the group $G$ has finite nonprosoluble length $l$. This means that $G$ has a normal series
\begin{equation*}
1=L_0\leqslant R_0 <  L_1\leqslant R_1< L_1  \leqslant \dots \leqslant R_{l}=G
\end{equation*}
in which each quotient $L_i/R_{i-1}$ is a (nontrivial) Cartesian product of non-abelian finite simple groups, and each quotient $R_i/L_{i}$ is prosoluble (possibly trivial).
We argue by induction on $l$. When $l=0$, the group $G$ is prosoluble, and the result follows by Theorem~\ref{t3}.

Now let $l\geq 1$. By Lemma~\ref{l-cart} each of the nonprosoluble factors $L_i/R_{i-1}$ is finite. In particular, the subgroup $L_1$ is virtually prosoluble, and therefore $\gamma _{\infty}(L_1)$ is finite by Corollary~\ref{c-virt}. The quotient $R_1/ \gamma _{\infty}(L_1)$ is prosoluble by Lemma~\ref{l-prosol-by-prosol}. Hence the  nonprosoluble length of $G/\gamma _{\infty}(L_1)$ is $l-1$. By the induction hypothesis we obtain  that $\gamma _{\infty}(G/\gamma _{\infty}(L_1))$ is finite, and therefore $\gamma _{\infty}(G)$ is finite. By Theorem~\ref{t2} the quotient $G/\gamma _{\infty}(G)$ is locally nilpotent, and the proof is complete.
\end{proof}

\section{Compact groups}
\label{s-comp}
In this section we prove the main Theorem~\ref{t1}
about compact groups with countable Engel sinks.  We use the structure theorems for compact groups and the results of the preceding section on profinite groups.  Parts of the proof are similar to the proof of the main result of \cite{khu-shu}, that is, Theorem~\ref{t4.1} on compact groups with finite Engel sinks. Rather than modifying the whole proof of that theorem, we are able to reduce the proof to the situation where all Engel sinks are finite, and then apply Theorem~\ref{t4.1}.

 By the well-known structure theorems (see, for example, \cite[Theorems~9.24 and 9.35]{hof-mor}), the connected component of the identity $G_0$ of a compact (Hausdorff) group $G$ is a divisible group such that $G_0/Z(G_0)$ is a Cartesian product of (non-abelian) simple compact Lie groups, while the quotient $G/G_0$ is a profinite group. (Recall that a group $H$ is said to be \emph{divisible} if for every $h\in H$ and every positive integer $k$ there is an element $x\in H$ such that $x^k=h$.)

We shall be using the following lemma from \cite{khu-shu}.

\begin{lemma}[{\cite[Lemma~5.3]{khu-shu}}] \label{l-eng}
Suppose that $G$ is a compact group in which every element has a finite Engel sink and the connected component of the identity $G_0$ is abelian. Then for every $g\in G$ and for any $x\in G_0$ we have
$$
[x,\,{}_kg]=1\quad \text{for some}\;\,k=k(x,g)\in {\mathbb N}.
$$
\end{lemma}

For compact groups with countable Engel sinks, we begin with eliminating simple Lie groups.

 \begin{lemma}\label{l-lie}
A non-abelian simple compact Lie group contains an element all of whose
Engel sinks are uncountable.
\end{lemma}

\begin{proof}
It is well known  that any non-abelian compact Lie group $G$ contains a subgroup isomorphic either to $SO_3 (\mathbb{R})$ or $SU_2( \mathbb{C})$ (see, for example, \cite[Proposition~6.46]{hof-mor}), and therefore in any case, a section isomorphic to $SO_3 (\mathbb{R})$. Since the property that every element has a countable Engel sink is  inherited by sections, it is sufficient to consider the case $G=SO_3 (\mathbb{R})$.

Consider the following elements of $SO_3 (\mathbb{R})$:
$$
a_\vartheta =\begin{pmatrix} \cos \vartheta &\sin \vartheta &0\\-\sin \vartheta  &\cos \vartheta &0\\0&0&1\end{pmatrix},\quad \vartheta\in  \mathbb{R},
$$
and
$$
g=\begin{pmatrix} -1&0&0\\0&1&0\\0&0&-1\end{pmatrix}.
$$
We have
\begin{align*}
[a_\vartheta ,g]=a_\vartheta ^{-1}a^g&=\begin{pmatrix} \cos (-\vartheta )&\sin (-\vartheta )&0\\-\sin (-\vartheta ) &\cos (-\vartheta )&0\\0&0&1\end{pmatrix}\cdot \begin{pmatrix} \cos \vartheta &-\sin \vartheta &0\\\sin \vartheta  &\cos \vartheta &0\\0&0&1\end{pmatrix}\\ &=\begin{pmatrix} \cos (-2\vartheta )&\sin (-2\vartheta )&0\\-\sin(-2\vartheta ) &\cos (-2\vartheta )&0\\0&0&1\end{pmatrix}=a_{-2\vartheta },
\end{align*}
and then by induction,
$$
[a_\vartheta ,\,{}_ng]=a_{(-2)^n\vartheta }.
$$
Therefore  any Engel sink of $g$ must contain, for every $\vartheta \in \mathbb{R}$, an element of the form $a_{(-2)^{n(\vartheta )}\vartheta }$ for some $n(\vartheta )\in {\mathbb N}$. Since for $\vartheta $ we can choose continuum elements of $\mathbb{R}$ that are linearly independent over  $\mathbb{Q}$,
any Engel sink of $g$ must be uncountable.
\end{proof}

The next lemma is a step towards proving that every element has a finite Engel sink.

\begin{lemma}\label{l-ab}
Suppose that $G$ is a compact group in which every element has a countable Engel sink. If $G$ has an abelian subgroup $A$ with locally nilpotent quotient $G/A$, then every element of $G$ has a finite Engel sink.
\end{lemma}

\begin{proof}
Since $G/A$ is locally nilpotent, for showing that an element $g\in G$ has a finite Engel sink we can obviously assume that $G=A\langle g\rangle$. Let $\{s_1,s_2,\dots \}$ be a countable Engel sink of $g$. By Lemma~\ref{l1}   there is an open subset  $U$ of $A$ and positive integers $m,n$ such that
$$
[u,\,{}_ng]=s_m \qquad \text{for all}\quad u\in U.
$$
The union of all translates $aU=\{au\mid u\in U\}$ over $a\in A$ is equal to $A$. Since $A$ is compact, it is equal to the union of finitely many such translates:
$$
A=a_1U\cup a_2U\cup \dots \cup a_kU.
$$
Then any element of $A$ has the form $a_iu$ for $u\in U$, and
$$
[a_iu,\,{}_ng]=[a_i,\,{}_ng][u,\,{}_ng]=[a_i,\,{}_ng]s_m,
$$
where we used the fact that $A$ is abelian. Hence the set
$$
S=\{[a_1,\,{}_ng]s_m,\dots ,[a_k,\,{}_ng]s_m\}
$$
is a finite Engel sink of $g$. Indeed, for any $x\in A$ and for any $k\geq n$ we have
\begin{equation*}
[x,\,{}_kg]=[[x,\,{}_{k-n}g],\,{}_ng]\in S.\tag*{\qedhere}
\end{equation*}
 \end{proof}

We are now ready to prove the main result.

\begin{theorem}\label{t5}
Suppose that $G$ is a compact group in which every element has a countable Engel sink. Then $G$ has a finite normal subgroup $N$ such that $G/N$ is locally nilpotent.
\end{theorem}

\begin{proof} In view of Lemma~\ref{l-lie}, the connected component of the identity $G_0$ is an abelian divisible normal subgroup.

\begin{lemma} \label{l-eng2}
For every $g\in G$ and for any $x\in G_0$ we have
$$
[x,\,{}_kg]=1\quad \text{for some}\;\,k=k(x,g)\in {\mathbb N}.
$$
\end{lemma}

\begin{proof}
We can obviously assume that $G=G_0\langle g\rangle$. The group $G_0\langle g\rangle$ satisfies the hypothesis of Lemma~\ref{l-ab} and therefore every element in it has a finite Engel sink. Then for any $x\in G_0$ we have $[x,\,{}_kg]=1$ for some $k=k(x,g)\in {\mathbb N}$  by Lemma~\ref{l-eng}.
\end{proof}

We proceed with the proof of Theorem~\ref{t5}. Applying Theorem~\ref{t4} to the profinite group $\bar G=G/G_0$ we obtain a finite normal subgroup $D$ with locally nilpotent quotient. Then every element $g\in\bar G$ has a finite smallest Engel sink $\bar{\mathscr E}(g)$ contained in $D$. Consider the subgroup generated by all such sinks:
$$
E=\langle \bar{\mathscr E}(g)\mid g\in \bar G\rangle\leq D.
$$
Clearly, $\bar{\mathscr E}(g)^h=\bar{\mathscr E}(g^h)$ for any $h\in\bar G$; hence  $E$ is a normal finite subgroup of $\bar G$. Note that $\bar G/E$ is also locally nilpotent by
Theorem~\ref{t-wz} as an Engel profinite group.

We now consider the action of  $\bar G$ by automorphisms  on $G_0$ induced by conjugation.

\begin{lemma}\label{l-central}
The subgroup $E$ acts trivially on $G_0$.
\end{lemma}

\begin{proof}
The abelian divisible group $G_0$ is a direct product $A_0\times\prod _pA_p$ of a torsion-free divisible group $A_0$ and divisible Sylow $p$-subgroups $A_p$ over various primes $p$. Clearly, every Sylow subgroup $A_p$ is normal in $G$.

First we show that $E$ acts trivially on each $A_p$. It is sufficient to show that for every $g\in \bar G$ every element  $z\in \bar{\mathscr E}(g)$ acts trivially on $A_p$. Consider the action of $\langle z, g\rangle$ on $A_p$. Note that $\langle z, g\rangle=\langle z^{\langle  g\rangle}\rangle\langle  g\rangle$, where $\langle z^{\langle  g\rangle}\rangle$ is  a finite $g$-invariant subgroup, since it is contained in the finite subgroup $E$. For any $a\in A_p$ we have $[a,\,{}_kg]=1$ for some $k=k(a,g)\in {\mathbb N}$  by Lemma~\ref{l-eng2}. Hence the subgroup
$$
\langle a^{\langle g\rangle}\rangle=\langle a,[a,g], [a,g,g],\dots \rangle
$$
is a finite $p$-group; note that this subgroup is  $g$-invariant.
 The images of $\langle a^{\langle g\rangle}\rangle$ under the action of elements of the finite group  $\langle z^{\langle  g\rangle}\rangle$  generate a finite $p$-group $B$, which is $\langle z, g\rangle$-invariant. It follows from Lemma~\ref{l-eng2} that $\langle z, g\rangle/C_{\langle z, g\rangle}(B)$
 must be a $p$-group. Indeed, otherwise there is a $p'$-element $y\in \langle z, g\rangle/C_{\langle z, g\rangle}(B)$ that acts non-trivially on the Frattini quotient $V=B/\Phi (B)$. Then $[[V,y],y]=[V,y]\ne 1$ and $C_{[V,y]}(y)=1$, whence  $[V,y]=\{[v,y]\mid v\in [V,y]\}$ and therefore also $[V,y]=\{[v,\,{}_ny]\mid v\in [V,y]\} $ for any $n$, contrary to Lemma~\ref{l-eng2}. Thus, $\langle z, g\rangle/C_{\langle z, g\rangle}(B)$ is a finite $p$-group. But since $z\in \bar{\mathscr E}(g)$, by Lemma~\ref{l-min} we have  $z=[z,{}_mg]$ for some $m\in {\mathbb N}$. Since a finite $p$-group is nilpotent, this implies that $z\in C_{\langle z, g\rangle}(B)$. In particular, $z$ centralizes $a$. Thus, $E$ acts trivially on $A_p$, for every prime~$p$.

 We now show that $E$ also acts trivially on the quotient $W=G_0/\prod _pA_p$ of $G_0$ by its torsion part. Note that $W$ can be regarded as a vector space over ${\mathbb Q}$. Every element $y\in E$ has finite order and therefore by Maschke's theorem   $W=[W,y]\times C_W(y)$ and $[W,y]=\{[w,\,{}_ny]\mid w\in [W,y]\} $  for any $n$. If $[W,y]\ne 1$, then this contradicts Lemma~\ref{l-eng2}.

 Thus, $E$ acts trivially both on $W$ and on  $\prod _pA_p$. Then any automorphism $\eta$ of $G_0$ induced by conjugation by $h\in E$ acts on every element $a\in A_0$ as $a^{\eta}=a^h=at$, where $t=t(a,h)$ is an element of finite order in $G_0$. Then  $a^{\eta ^i}=at^i$, and therefore the order of $t$ must divide the order of $\eta$.

 Assuming the action of $E$ on $G_0$ to be non-trivial, choose an element $h\in E$  acting on $G_0$ as an automorphism $\eta$ of some prime order $p$. Then there is $a\in A_0$ such that $a^h=as$, where $s\in A_p$ has order $p$. There is an element $a_1\in A_0$ such that $a_1^{p}=a$. Then $a_1^h=a_1s_1$, where $s_1^{p}=s$. Thus, $|s_1|=p^{2}$, and therefore  $p^{2}$ divides the order of $\eta $. We arrived  at a contradiction with $|\eta |=p$.
\end{proof}

We now finish the proof of Theorem~\ref{t5}. Let $F$ be the full inverse image of $E$ in $G$. Then we have normal subgroups $G_0\leq F\leq G$ such that $G/F$ is locally nilpotent, $F/G_0$ is finite, and $G_0$ is contained in the centre of $F$ by Lemma~\ref{l-central}. Since $F$ has centre of finite index, the derived subgroup $F'$ is finite by Schur's theorem \cite[Satz~IV.2.3]{hup}. The quotient $G/F'$ is an extension of an abelian subgroup by a locally nilpotent group. Hence every element of $G/F'$ has a finite Engel sink by Lemma~\ref{l-ab}. By Theorem~\ref{t4.1} the group  $G/F'$ has a finite normal subgroup with locally nilpotent quotient. The full inverse image of this subgroup is a required finite normal subgroup $N$ such that $G/N$ is locally nilpotent.
The proof of  Theorem~\ref{t5} is complete.
 \end{proof}

 \section*{Acknowledgements}
The authors thank John Wilson for stimulating discussions.

The work is supported by Mathematical Center in Akademgorodok, the agreement with Ministry of Science and High Education of the Russian Federation no.~075-15-2019-1613. The first author thanks CNPq-Brazil and the University of Brasilia for support and hospitality that he enjoyed during his visit to Brasilia in 2019.  The second author was supported by FAPDF and CNPq-Brazil.


\begin{thebibliography}{99}
\bibitem{bou} N. Bourbaki, Elements of Mathematics. Lie Groups
and Lie Algebras. Part I: Chapters 1--3, Hermann, Paris, Addison-Wesley, Reading, MA, 1975.

\bibitem{br-ge}
E. Breuillard and T. Gelander,
A topological Tits alternative,  \textit{Ann. Math. (2)} {\bf  166}, no.~2 (2007), 427--474.

\bibitem{car} R. W. Carter, \textit{Simple groups of Lie type}, John Wiley \& Sons, London, 1972.

\bibitem{dms1} E. Detomi, M. Morigi, P. Shumyatsky, { Bounding the exponent of a verbal subgroup}, \emph{Annali  Mat.} \textbf{193} (2014), 1431--1441.

\bibitem{fei-tho} W. Feit and J. G. Thompson, Solvability of groups of odd order, \textit{Pac. J. Math.} \textbf{13} (1963), 775--1029.

\bibitem{ha-hi}  P. Hall and G. Higman, The $p$-length of a $p$-soluble group and reduction
theorems for Burnside's problem, \emph{Proc. London Math. Soc. (3)} {\bf
 6} (1956), 1--42.

\bibitem{hof-mor} K. H. Hofmann and S. A. Morris, \textit{The Structure of Compact Groups}, De Gruyter, Berlin, 2006.

\bibitem{hup} B. Huppert, \textit{Endliche Gruppen}. I,
Springer, Berlin, 1967.

\bibitem{hup-bla} B. Huppert and N. Blackburn, \textit{Finite Groups}. II,
Springer, Berlin, 1982.

\bibitem{kel} J. L. Kelley, \textit{General topology}, Grad. Texts in Math., vol.~27, Springer, New York,  1975.

\bibitem{khu-shu131} E. I. Khukhro and P. Shumyatsky, Nonsoluble and non-$p$-soluble length of finite groups, \textit{Israel J. Math.}, \textbf{207} (2015), 507--525.

    \bibitem{khu-shu132} E. I. Khukhro and P. Shumyatsky, Words and pronilpotent subgroups in profinite groups, \emph{J.~Austral. Math. Soc.} {\bf  97}, no.~3 (2014), 343--364.

\bibitem{khu-shu} E. I. Khukhro and P. Shumyatsky, Almost Engel compact groups, \emph{J.~Algebra} {\bf  500} (2018), 439--456.

    \bibitem{laz} M. Lazard, Groupes analytiques $p$-adiques, \textit{Publ. Math. Inst. Hautes \'Etudes
Sci.} \textbf{26} (1965), 389--603.

\bibitem{med} Yu. Medvedev,  On compact Engel groups,  Israel J. Math. {\bf 185} (2003), 147--156.

\bibitem{rei} C. D. Reid, Local Sylow theory of
totally disconnected, locally compact groups, \textit{J.~Group Theory} \textbf{16} (2013), 535--555.

\bibitem{rib-zal} L. Ribes and P. Zalesskii, \emph{Profinite groups},
Springer, Berlin, 2010.

\bibitem{rob} D. J. S. Robinson, A course in the theory of groups, Springer, New York, 1996.

    \bibitem{rod-shu} S. Rodrigues and P. Shumyatsky, Exponent of a finite group admitting a
coprime automorphism, \textit{to appear in J.~Pure Appl. Algebra}
(available online 5 March 2020, 106370).

\bibitem{sha} A. Shalev, Polynomial identities in graded group rings, restricted Lie algebras and $p$-adic analytic groups, \textit{Trans. Amer. Math. Soc.} \textbf{337}, no.~1 (1993), 451--462.


\bibitem{shu00}   P. Shumyatsky, On pro-$p$ groups admitting a fixed-point-free automorphism, \textit{J.~Algebra} \textbf{228}, no.~1 (2000), 357--366.


\bibitem{68}
 P. Shumyatsky, Commutators in residually finite groups, \emph{Israel J. Math.} {\bf 182} (2011), 149--156.

 \bibitem{austral}
P. Shumyatsky,  On the exponent of a verbal subgroup in a finite group, \emph{J.~Austral. Math. Soc.}, {\bf 93} (2012), 325--332.

\bibitem{wil83} J. S. Wilson, On the structure of compact torsion groups, {\it
Monatsh. Math.} {\bf 96} (1983), 57--66.

\bibitem{wil} J. S. Wilson, \emph{Profinite groups} (Clarendon Press, Oxford, 1998).

\bibitem{wi-ze} J. S. Wilson and E. I. Zelmanov, Identities for Lie algebras of pro-$p$ groups, \emph{J. Pure Appl. Algebra} {\bf 81}, no.~1 (1992), 103--109.


\bibitem{ze92} E. Zelmanov, Nil rings and periodic groups, \textit{Korean Math. Soc. Lecture
Notes in Math.}, Seoul, 1992.

\bibitem{ze95}    E. Zelmanov, Lie methods in the theory of nilpotent groups, in: \textit{Groups'\,93 Galaway/St Andrews}, Cambridge Univ. Press, Cambridge, 1995, 567--585.

\bibitem{ze17} E. Zelmanov,
Lie algebras and torsion groups with identity, \textit{J.~Comb. Algebra} \textbf{1}, no.~3 (2017), 289--340.

\end{thebibliography}
\end{document}